\documentclass[reqno]{amsart}
\usepackage{amsmath,amsfonts,amssymb,graphics,graphicx,latexsym,amscd,subfigure,hyperref, mathtools}
\usepackage[usenames]{xcolor}
\usepackage[retainorgcmds]{IEEEtrantools}

\newtheorem{thm}{Theorem}[section]

\newtheorem{question}[thm]{Question}
\newtheorem{lemma}[thm]{Lemma}

\newtheorem{exa}[thm]{Example}

\newcommand{\bbR}{\mathbb{R}}
\newcommand{\bbT}{\mathbb{T}}
\newcommand{\bbC}{\mathbb{C}}
\newcommand{\bbZ}{\mathbb{Z}}
\newcommand{\bbP}{\mathbb{P}}

\newcommand{\cp}{P}
\newcommand{\s}{g}
\newcommand{\spot}{\text{Spot}}

\newcommand{\hess}{\operatorname{Hess}}
\newcommand{\interior}{\operatorname{int}}
\newcommand{\Lap}{\triangle}
\newcommand{\diagentry}[1]{\mathmakebox[1.8em]{#1}}
\newcommand{\xddots}{%
  \raise 4pt \hbox {.}
  \mkern 6mu
  \raise 1pt \hbox {.}
  \mkern 6mu
  \raise -2pt \hbox {.}
}

\def\det{\mathop{\rm det}\nolimits}

\newcommand\Cinf{\mathcal{C}^\infty_0}
\newcommand{\F}{\mathcal F}

\begin{document}

\title[]{Recovering $U(n)$-invariant toric K\"ahler metrics on $\bbC\bbP^n$ from the torus equivariant spectrum}

\author{Tom\' as A. Reis}
\address{Departamento de F\' isica, Instituto Superior T\'{e}cnico, Av. Rovisco Pais, 1049-001 Lisboa, Portugal}
\email{tomas.a.reis@tecnico.ulisboa.pt}

\author{Rosa Sena-Dias}
\address{Centro de An\'alise Matem\'atica, Geometria e Sistemas Din\^amicos, Departamento de Matem\'{a}tica, Instituto Superior T\'{e}cnico, Av. Rovisco Pais, 1049-001 Lisboa, Portugal}
\email{rsenadias@math.ist.utl.pt}
\thanks{TR was partially supported by the Gulbenkian Foundation through a ``Novos talentos em Matem\' atica" grant. RSD was partially supported by FCT/Portugal through projects PEst-OE/EEI/LAOO9/2013, EXCL/MAT-GEO/0222/2012 and PTDC/MAT/117762/2010}

\date{}

\begin{abstract}
In this note we prove that toric K\"ahler metrics on complex projective space which are also $U(n)$-invariant are determined by their \emph{equivariant} spectrum i.e. the list of eigenvalues of the Laplacian together with weights of the torus representation on the eigenspaces.
\end{abstract}

\maketitle

\section{Introduction}

Given a compact Riemannian manifold, its Laplacian is a self-adjoint non-negative operator whose spectrum is a discrete subset of $\bbR_0^+$.  A classical question in Riemannian Geometry is ``Does the spectrum of the Laplacian on a Riemannian manifold determine the Riemannian manifold?". Although the answer to this question is no (see \cite{m}), there are several special metrics or classes of metrics that are known to be determined by the spectrum of the Laplacian and there are also known classes of Riemannian manifolds among which the spectrum determines the underlying manifold or the metric. In fact, this question is one of the main objects of study in spectral geometry. For instance the Fubiny-Study metric on $\bbC\bbP^n$ is determined by its spectrum among all K\"ahler metrics if $n\leq 6$ (\cite{t}). Recently there has been an interest in studying inverse spectral problems on Riemannian manifolds with an isometric group action for the so-called \emph{equivariant spectrum} which is roughly the spectrum together with a record of the induced action of the group on the eigenspaces. We will give more details on this ahead. In particular, there have been some equivariant inverse spectral results for toric K\"ahler manifolds. Toric manifolds are symplectic manifolds with an effective, Hamiltonian action of a torus of maximal dimension. These can be given  K\"ahler structures and the question is to what extent is the toric K\"ahler manifold determined by the equivariant spectrum of the Laplacian of the metric associated to the toric K\"ahler structure (see \cite{a}, \cite{dgs1}, \cite{dgs3}). 

The question we want to address here is an equivariant inverse spectral question for \emph{metrics} i.e. the underlying manifold is fixed, in our case it is $\bbC \bbP^n$, and we want to know if the Riemannian metric on it is determined by the equivariant spectrum of its Laplacian. This sort of question is known to be hard and there are not many results in this direction. In \cite{z}, \cite{dgs3}, \cite{dms} there are a few results concerning $S^1$-invariant metrics on $S^2$. The results in \cite{dgs3} and \cite{dms} concern the equivariant spectrum while those in \cite{z} concern the spectrum. In this note we will use some of the general results derived in \cite{dgs3} to prove a positive inverse equivariant spectral result for $U(n)$-invariant metrics on  $\bbC \bbP^n, \,n\geq 2$ . Namely we prove the following.

\begin{thm}\label{main}
Let $n\geq 2$. An $U(n)$-invariant  toric K\"ahler metric on  $\bbC \bbP^n$ is uniquely determined by the torus equivariant spectrum of the Laplacian.
\end{thm}
The case $n=1$ corresponds to $S^1$-invariant metric on $\bbC \bbP^1=S^2$. Inverse spectral results for $S^1$-invariant metric on $S^2$ using the equivariant spectrum were previously obtained in \cite{dgs3} and \cite{dms}.

The main ingredient involved in the proof of the above theorem is a spectral invariant that can be derived from the spectral invariants discovered in \cite{dgs3}. Some of the invariants were made explicit in the case of toric manifolds. Toric manifolds are symplectic with a Hamiltonian action from a torus. Because it is Hamiltonian, the action admits a moment map whose image we call the moment polytope. It has been shown (see \cite{g}) that toric manifolds admit compatible torus-invariant K\"ahler structures. These are parametrized by functions on the open moment polytope called symplectic potentials (see \cite{a}). We will give more details on this ahead.  $\bbR_0^+$ denotes the set of non-negative real numbers. 

\begin{thm}\label{thm_spec_inv}
Let $X^{2n}$ be a toric manifold with moment polytope $\cp$. Let $X$ be endowed with a toric K\"ahler structure determined by a symplectic potential $\s:\interior({\cp})\rightarrow \bbR$. Then the following quantity is determined by the equivariant spectrum of the Riemannian metric associated with the K\"ahler structure on $X$
\begin{equation}\label{spec_inv}
\int_{\cp}\rho(\alpha^t \hess(\s)\alpha)\sqrt{\det  \hess(\s)}dx
\end{equation}
for all $\alpha\in \bbR^n$ and $\rho\in \Cinf(\bbR_0^+)$.
\end{thm}
The proof of this theorem follows from results and techniques in \cite{dgs3} but since the above was never stated in \cite{dgs3}, we give a proof here for the sake of completeness using the results in \cite{dgs3}.

{\noindent \textbf{Acknowledgements.} This work was carried out in the context of the ``Novos talentos em Matem\' atica" program. We would like to thank the Gulbenkian Foundation and the program organizers for all their support. }

\section{Background}
The goal of this section is to quickly review the necessary background involved in the statement and the proof of our main theorem \ref{main}. We will be brief and will assume some familiarity with both spectral geometry and toric geometry. The reader is referred to \cite{dgs3} for more background on equivariant spectral geometry and \cite{g}, \cite{a} for more on toric K\"ahler geometry.
\subsection{Spectral Geometry}
Let $(X,ds^2)$ be a Riemannian compact manifold. There is a well defined Laplace operator which is self-adjoint and positive. The set of eigenvalues of the Laplace operator counted with multiplicities is called the \emph{spectrum} of the Riemannian manifold. It is well known that 
\begin{itemize}
\item The spectrum is a discrete subset of $\bbR_0^+$. 
\item Each eigenvalue of the Laplacian has a {finite}-dimensional eigenspace.
\end{itemize}

Let $G$ be a group acting on $(X,ds^2)$ by isometries. Then, because each group-induced isometry commutes with the Laplace operator, $G$ induces an action on each eigenspace i.e. a $G$-representation. The set of all eigenvalues of the Laplace operator together with the corresponding isomorphism class of the $G$-representation  on the corresponding eigenspace is called the \emph{$G$-equivariant spectrum}. 

We will mainly be concerned with the case where $G$ is a torus $\bbT^n$. In this case, isomorphism classes of representations are given by a list of integers called the weights which correspond to eigenvalues of the $\bbT^n$ action on the eigenspace. To be more precise, let $m$ be an integer and $w=(w_1,\cdots w_m),\, w_i\in \bbZ^n, \,i=1,\cdots, m$. Then we can define a real representation of $\bbT^n$ on $\bbC^m$,  $\psi_w:\bbT^n\rightarrow \text{GL}(m)$ to be the group homeomorphism given by 
$$
\psi_w(e^{\sqrt{-1}\theta})=
\begin{pmatrix}
&\diagentry{e^{\sqrt{-1}\theta\cdot w_1}}\\
&&\diagentry{\xddots}\\
&&&\diagentry{e^{\sqrt{-1}\theta\cdot w_m}}&\\
\end{pmatrix},
 \, \theta\in \bbR^n.
$$
Every $\bbT^n$ representation is isomorphic to one of the $\psi_w$ above and the $w$ are called the weights of the representation. 

If a Riemannian manifold admits an isometric torus action, its equivariant spectrum is the spectrum together with the weights of the torus representation on each eigenspace. 

In \cite{dgs3}, the authors define a family of equivariant spectral measures associated to any isometric torus action on a Riemannian manifold. More precisely, let $(X,ds^2,\bbT^n)$ be a Riemannian manifold with an isometric torus action.

There is a well defined spectral measure, $\mu(\alpha,\hbar)$ that ``counts" $\frac{\alpha}{\hbar}$-equivariant eigenfunctions of the operator $\hbar^2 \Lap$, where $\hbar$ is a small real number and $\Lap$ is the Laplacian of the Riemannian manifold. The main result in \cite{dgs3} is an asymptotic expansion for $\mu(\alpha,\hbar)$ as $\hbar$ tends to zero. The leading order term of this asymptotic can be written down explicitly. Since the measure depends only the $\bbT^n$-equivariant spectrum, the leading order term in the asymptotic expansion is an equivariant spectral invariant meaning that it only depends on the equivariant spectrum of the Riemannian manifold and not on the full manifold. We will not write this equivariant spectrum invariant here in full generality. Instead we will say that in some cases it assumes a very simple expression. One such case is that of toric manifolds. This is also the main application in \cite{dgs3}.
\subsection{Toric Geometry}
A toric manifold is a symplectic manifold $(X^{2n},\omega)$ admitting an effective  Hamiltonian $\bbT^n$-action. Such an action admits a moment map $\phi:X\rightarrow \bbR^n$, where we have identified the Lie algebra of $\bbT^n$ with $\bbR^n$.  It is a well known theorem, due to Atiyah-Bott and Guillemin-Sternberg, that the image of this moment map is convex and in fact it is convex polytope of Delzant type (see \cite{g} for a definition). A Delzant polytope can be written as
\begin{equation}\label{delzant_polytope}
\{x\in\bbR^n: l_i(x)=\langle x,u_i\rangle-\lambda_i\geq0, \,i=1,\cdots d\},
\end{equation}
where $u_i$ are vectors in $\bbR^n$ normal to the facets of the Delzant polytope. The Delzant condition can be expressed in terms of the $u_i$. This Delzant polytope $\cp$, determines $(X^{2n},\omega)$ up to equivariant symplectomorphism. 

There is an open dense set on the toric manifold $X$ which can be identified with with $\cp\times \bbT^n$ in an equivariant way ($\bbT^n$ acts on $\cp\times \bbT^n$  on the second factor via multiplication). This identification gives coordinates on this open dense set of $X$. Such coordinates are called \emph{action-angle} coordinates (see \cite{a} for more details).

In this paper we are only concerned with the simplest example of a toric manifold, namely that of $\bbC\bbP^n$. 
\begin{exa}
Let $\bbC\bbP^n=S^{2n+1}/S^1$ be complex projective space endowed with the usual Fubiny-Study metric. A point in $\bbC\bbP^n$ has homogeneous coordinates $[z_0:z_1\cdots:z_n]$ where $\sum_{i=0}^n|z_i|^2=1$ so that {$[e^{\sqrt{-1}\theta}z_0:e^{\sqrt{-1}\theta}z_1\cdots:e^{\sqrt{-1}\theta}z_n]=[z_0:z_1\cdots:z_n]$}, for any $\theta\in \bbR$. Consider the action of $(S^1)^n=\bbT^n$ on $\bbC\bbP^n$ given by
$$
(e^{\sqrt{-1}\theta_1},\cdots,e^{\sqrt{-1}\theta_n})\cdot[z_0:z_1\cdots:z_n]=[z_0:e^{\sqrt{-1}\theta_1}z_1\cdots:e^{\sqrt{-1}\theta_n}z_n], 
$$
where $(\theta_1,\cdots,\theta_n)\in \bbR^n, [z_0:z_1\cdots:z_n]\in \bbC\bbP^n$. The moment map is $\phi:\bbC\bbP^n\rightarrow \bbR^n$ given by
$$
\phi([z_0:z_1\cdots:z_n])=(|z_1|^2,\cdots,|z_n|^2), \,[z_0:z_1\cdots:z_n]\in \bbC\bbP^n,
$$
and the moment map image is 
$$
\cp=\{(x_1,\cdots,x_n)\in\bbR^n: x_i\geq0,\,i=1,\cdots,n,\,\sum_{i=1}^nx_1\leq1\},
$$ i.e. the standard simplex in $\bbR^n$. The action-angle coordinates of $[z_0:z_1\cdots:z_n]$ are $(|z_1|^2,\cdots,|z_n|^2)$ and $(\frac{z_1/|z_1|}{z_0/|z_0|},\cdots,\frac{z_n/|z_n|}{z_0/|z_0|})$ on the open subset of $\bbC\bbP^n$ where $z_0\ne0$.
\end{exa}

Toric K\"ahler metrics on a given toric manifold $(X,\omega)$ are K\"ahler metrics compatible with $\omega$ which are invariant under the torus action. In \cite{g}, Guillemin shows that this set of toric K\"ahler metrics on any given toric manifold $(X,\omega)$ in non-empty by constructing an explicit example of such a metric namely the Guillemin metric. In action-angle coordinates $(x,\theta)$ this metric is given by
$$
\begin{pmatrix}
\phantom{-}\hess(\s_0) & \vdots & 0\  \\
\hdotsfor{3} \\
\phantom{-}0 & \vdots & \hess^{-1}(\s_0)
\end{pmatrix}
$$
where $\s_0$ can be explicitly written down in terms of the defining function for the moment polytope of $X$, $\cp$. More explicitly, $\s_0=\sum_{i=1}^dl_i\log(l_i)-l_i$ where $\cp$ is written as in (\ref{delzant_polytope}). This function is called the Guillemin potential of $\cp$. It has been shown by Abreu that all toric K\"ahler metrics can be explicitly written down in terms of a function on the moment polytope of $X$ as above i.e. toric K\"ahler metrics are given in action angle coordinates by 
$$
\begin{pmatrix}
\phantom{-}\hess(\s) & \vdots & 0\  \\
\hdotsfor{3} \\
\phantom{-}0 & \vdots & \hess^{-1}(\s)
\end{pmatrix}
$$
where $\s:\interior(\cp)\rightarrow\bbR$  is called the \emph{symplectic potential} of the toric K\"ahler metric and it satisfies the following properties
\begin{itemize}
\item $\s$ is strictly convex,
\item $\s-\s_0$ extends to a smooth function on $\cp$,
\item There exists a function $\delta\in \Cinf(\cp)$ such that $\det \hess \s=\left(\delta\prod_{i=1}^dl_i\right)^{-1}$.
\end{itemize}
We denote the set of all such symplectic potentials by $\spot(\cp)$.
\begin{exa}
In the case of $\bbC\bbP^n$ the Guillemin potential is given by
$$
\s_0=\sum_{i=1}^n\left(x_i\log(x_i)-x_i\right)+ \left(1-\sum_{i=1}^nx_i\right)\log\left(1-\sum_{i=1}^nx_i\right)-\left(1-\sum_{i=1}^nx_i\right).
$$
\end{exa}
The space $\bbC\bbP^n$ also admits an action of $U(n)$. In fact, it is clear from the description of $\bbC\bbP^n$ using homogeneous coordinates that $\bbC\bbP^n$ admits an action of $U(n+1)$. This action is non-effective and has an $S^1$ (the diagonal $S^1$) acting trivially. Since $U(n)$ embeds into $U(n+1)$ (in several ways) we get a $U(n)$-action. Toric K\"ahler metric which are also invariant under the action of $U(n)$ are the ones whose symplectic potentials are of the form $\s_0(x)+h\left(\sum_{i=1}^nx_i\right)$. They are also cohomogeneity one metrics and they play an important role in K\"ahler geometry (see \cite{d}, \cite{gz} or \cite{a2} for instance).

\section{Equivariant spectral invariants on toric manifolds}
The goal of this section is to prove Theorem \ref{thm_spec_inv}. This follows by manipulating the equivariant spectral invariants for toric manifolds that were written down in \cite{dgs3}. We start by stating the result for toric manifolds in [\cite{dgs3} Lemma 6.14].
\begin{thm}[Dryden-Guillemin-Sena-Dias] Let $X^{2n}$ be a toric manifold with moment polytope $\cp$. Suppose $X$ is endowed with a toric K\"ahler metric with symplectic potential $\s$. Then 
$$
\int_{\cp\times \bbR^{n}}\rho(\alpha^t\hess(\s)\alpha+\xi^t\hess^{-1}(\s)\xi)dxd\xi
$$
is determined by the equivariant spectrum for all generic $\alpha \in \bbR^n$ and for all $\rho \in \Cinf(\bbR_0^+)$.
\end{thm}
We are now in a position to prove Theorem \ref{thm_spec_inv}.
\begin{proof}(of Theorem \ref{thm_spec_inv})
We start by noting that by continuity 
$$
\int_{\cp\times \bbR^{n}}\rho(\alpha^t\hess(\s)\alpha+\xi^t\hess^{-1}(\s)\xi)dxd\xi
$$ is actually determined by the equivariant spectrum for \emph{all} $\alpha \in \bbR^n$ and for all $\rho \in \Cinf(\bbR_0^+)$. 
Let $\alpha$ be any fixed vector in $\bbR^n$ and let $F \in \Cinf(\bbR_0^+)$.
$$
\int_{\cp\times \bbR^{n}}F(\alpha^t\hess(\s)\alpha+\xi^t\hess^{-1}(\s)\xi)dxd\xi
$$ 
is determined. We want to simplify this integral using a change of variables. First note that because $\hess^{-1}(\s)$ is symmetric there is a unitary matrix $S(x)$ which diagonalizes $\hess^{-1}(\s)$ i.e. such that
$$
S(x)\hess^{-1}(\s)(x)S^{-1}(x)=\begin{pmatrix}
\diagentry{\lambda_1(x)}\\
&\diagentry{\xddots}\\
&&\diagentry{\lambda_n(x)}\\
\end{pmatrix}
$$
where $\lambda_1(x),\cdots, \lambda_n(x)$ are the positive eigenvalues of $\hess^{-1}(\s)(x)$. Set $y=S\xi$. Because $S$ is unitary we can change variables in the above integral to get 
$$
\int_{\cp\times \bbR^{n}}F(\alpha^t\hess(\s)\alpha+\sum_{i=1}^n\lambda_i(x)y_i^2)dxdy.
$$
Now set $u_i=\sqrt{\lambda_i}y_i, \, i=1,\cdots, n$ and change variables again to conclude that 
$$
\int_{\cp\times \bbR^{n}}\frac{F(\alpha^t\hess(\s)\alpha+\sum_{i=1}^nu_i^2)}{\sqrt{\prod_{i=1}^n\lambda_i}}dxdu,
$$
is determined by the equivariant spectrum. Note that $\prod_{i=1}^n\lambda_i=\det \hess^{-1}(\s)(x)$ so that the above can rewritten as 
$$
\int_{\cp\times \bbR^{n}}F(\alpha^t\hess(\s)\alpha+\sum_{i=1}^nu_i^2)\sqrt{\det \hess(\s)}dxdu.
$$
Next set $r$ to be the radial coordinate in $\bbR^n$ so that $\sum_{i=1}^nu_i^2=r^2$. We see that 
$$
\int_{\cp}\left(\int_0^\infty F(\alpha^t\hess(\s)\alpha+r^2)r^{n-1}dr\right)\sqrt{\det \hess(\s)}dx,
$$
is determined by the equivariant spectrum. Because $F$ has compact support, the function $\int_0^\infty F(\alpha^t\hess(\s)\alpha+r^2)r^{n-1}dr$ is a well defined function of $\alpha^t\hess(\s)\alpha$. We can write 
$$
\rho(t)=\int_0^\infty F(t+r^2)r^{n-1}dr,
$$
for some $\rho \in \Cinf(\bbR_0^+)$. To prove our theorem, it will be enough to show that, given any $\rho\in\Cinf(\bbR_0^+)$, there is an $F \in\Cinf(\bbR_0^+)$, such that 
$$
\rho(t)=\int_0^\infty F(t+r^2)r^{n-1}dr.
$$
We will do this by induction on $n$. 
\begin{itemize}
\item First we treat the case $n=1$. We would like to show that, given any $\rho\in \Cinf(\bbR_0^+)$, there is $F\in \Cinf(\bbR_0^+)$, such that 
$$
\rho(t)=\int_0^\infty F(t+r^2)dr.
$$
Set $\nu=\frac{1}{t+r^2}$. Then 
$$
\int_0^\infty F(t+r^2)dr=\frac{1}{2\sqrt{t}}\int_0^{\frac{1}{t}}\frac{F\left(\frac{1}{\nu}\right)d\nu}{\nu\sqrt{\nu}\sqrt{\frac{1}{t}-\nu}}.
$$
The Abel transform of a continuous function $f$ on $\bbR_0^+$ is defined to be
\begin{equation}\label{Abel}
J(f)(x)=\int_0^x\frac{f(\nu)d\nu}{\sqrt{x-\nu}}, \, x\in \bbR^+.
\end{equation}
Now given $F\in\Cinf(\bbR_0^+)$ set
$$
K(F)(\nu)=\frac{F\left(\frac{1}{\nu}\right)}{\nu\sqrt{\nu}}.
$$
We see that 
$$
\int_0^\infty F(t+r^2)dr=\frac{1}{2\sqrt{t}}J(K(F))\left(\frac{1}{t}\right).
$$
and we are trying to solve
$$
\rho(t)=\frac{J(K(F))\left(\frac{1}{t}\right)}{2\sqrt{t}},
$$
which gives
{$$
J(K(F))(t)=\frac{2\rho(\frac{1}{t})}{\sqrt{t}}.
$$}
The upshot is that the Abel transform is known to be invertible (see \cite{gs} for more details on the Abel transform). More precisely
$$
f(x)=\frac{dJ(J(f))(x)}{dx}
$$
for any continuous function $f$, on $\bbR_0^+$. Therefore it follows that 
$$
F(t)=\frac{2}{t\sqrt{t}}\frac{d}{dt} \left[J\left(\frac{\rho\left(\frac{1}{t}\right)}{\sqrt{t}}\right)\right]\left(\frac{1}{t}\right).
$$
Because $\rho$ has compact support, $J\left(\frac{\rho\left(\frac{1}{t}\right)}{\sqrt{t}}\right)$ vanishes in a neighborhood of zero and so $\left[J\left(\frac{\rho\left(\frac{1}{t}\right)}{\sqrt{t}}\right)\right]\left(\frac{1}{t}\right)$ has compact support. This implies that $F$ as defined above has compact support.

\item Next we treat the case $n\geq2$. Assume that for all $l< n$ and all $\rho\in\Cinf(\bbR_0^+)$ there is an $F \in\Cinf(\bbR_0^+)$ such that 
$$
\rho(t)=\int_0^\infty F(t+r^2)r^{l-1}dr.
$$
Now given $\rho\in\Cinf(\bbR_0^+)$, we would like to find $F \in\Cinf(\bbR_0^+)$ such that 
$$
\rho(t)=\int_0^\infty F(t+r^2)r^{n-1}dr.
$$
We have 
\begin{IEEEeqnarray*}{c}
\int_0^\infty F(t+r^2)r^{n-1}dr=\int_0^\infty F(t+r^2)r\cdot r^{n-2}dr\\
=\frac{1}{2}\int_0^\infty \frac{d}{dr}\left(f(t+r^2)\right)r^{n-2}dr,
\end{IEEEeqnarray*}
where $f$ is given by $f(t)=\int_0^tF(\nu)d\nu-\int_0^\infty F(\nu)d\nu$. Because $F$ has compact support, it has finite integral on $\bbR_0^+$ and $f$ also has compact support. Integrating the above expression by parts we see that 
$$
\int_0^\infty F(t+r^2)r^{n-1}dr=\begin{cases}
&-\frac{n-2}{2}\int_0^\infty f(t+r^2)r^{n-3}dr,\, n\geq3\\
&\frac{-f(t)}{2},\, n=2.
\end{cases}
$$
If $n=2$ we can simply set $\rho(t)=-f(t)/2$ i.e. $F(t)=-2\rho'(t)$ which has compact support. If $n\geq3$ we have $\rho(t)= -\frac{n-2}{2}\int_0^\infty f(t+r^2)r^{n-3}dr$ which by induction admits a unique solution $f\in \Cinf(\bbR_0^+)$. Since $F=f'$ the result follows.
\end{itemize}
\end{proof}

\section{Recovering $U(n)$-invariant metrics on $\bbC \bbP^n$ from the spectral invariant }
In this section our goal is to use Theorem \ref{thm_spec_inv} to prove our main result Theorem \ref{main}. We know from Theorem \ref{thm_spec_inv} that the equivariant spectrum of the metric whose symplectic potential is $\s$ determines
$$
\int_{\cp}\rho(\alpha^t \hess(\s)\alpha)\sqrt{\det  \hess(\s)}dx, \, \forall \rho\in \Cinf(\bbR_0^+), \, \alpha\in \bbR^n.
$$
It is then natural to ask the following question
\begin{question}
Let $\cp$ be a Delzant polytope. We define $\F:\spot(\cp)\rightarrow \mathcal{C}^\infty(\bbR^n\times \Cinf(\bbR_0^+))$ by
$$
\F(\s)[\alpha,\rho]=\int_{\cp}\rho(\alpha^t \hess(\s)\alpha)\sqrt{\det  \hess(\s)}dx.
$$
Is $\F$ injective on $\spot(\cp)$ or on a ``reasonably" large subset of $\spot(\cp)$?
\end{question}
We are able to answer this question under some very strong assumptions namely when $\cp$ is the the standard simplex in $\bbR^n$ i.e. the moment polytope of $\bbC \bbP^n$ 
$$
\cp=\{x=(x_1,\cdots, x_n)\in \bbR^n:  x_i\geq i=1,\cdots n, 1-\sum_{i=1}^nx_i\geq 0\},
$$
and $\F$ is restricted to the subset of $\spot(\cp)$ corresponding to $U(n)$-invariant metrics. More precisely $\s$ is of the form
{
\begin{equation}\label{un_potential}
\sum_{i=1}^n\left(x_i\log(x_i)-x_i\right)+\left(1-\sum_{i=1}^nx_i\right)\log\left(1-\sum_{i=1}^nx_i\right)-\left(1-\sum_{i=1}^nx_i\right)+h\left(\sum_{i=1}^nx_i\right)
\end{equation}
}
where $h$ is a smooth function on $\bbR_0^+$. To determine $\s$ it is enough to determine $h$ and so we need to express $\F(\s)$ in terms of $h$ alone.
\begin{lemma}
Let $\cp$ be the standard simplex as above and let $\s \in \spot(\cp)$ be of the form (\ref{un_potential}). Then, given $\alpha \in \bbR^n$, and $\rho \in \Cinf(\bbR_0^+)$,
\begin{IEEEeqnarray}{lcr}\label{Fh}
&\F(\s)=&\int_{\cp}\rho\left( \sum_{i=1}^n\frac{\alpha_i^2}{x_i}+\left(\frac{1}{1-\sum_{i=1}^nx_i}+h''\left(\sum_{i=1}^nx_i\right)\right)\left(\sum_{i=1}^n\alpha_i\right)^2\right)\\ \nonumber
&&\sqrt{\frac{\frac{1}{1-\sum_{i=1}^nx_i}+\sum_{i=1}^nx_ih''\left(\sum_{i=1}^nx_i\right)}{\prod_{i=1}^nx_i}}dx.
\end{IEEEeqnarray}
\end{lemma}
\begin{proof}
This lemma is a simple consequence of the fact that if $\s$ is of the form (\ref{un_potential}) then its Hessian is given by
$$
\hess(\s)=
\begin{pmatrix}
\diagentry{\frac{1}{x_1}}\\
&\diagentry{\xddots}\\
&&\diagentry{\frac{1}{x_n}}\\
\end{pmatrix}+ 
\left(\frac{1}{1-\sum_{i=1}^nx_i}+h''\left(\sum_{i=1}^n x_i\right)\right)\begin{pmatrix}
1\cdots1\\
\vdots\\
1\cdots 1
\end{pmatrix}.
$$
It follows immediately that 
$$
\alpha^t\hess(\s)\alpha=\sum_{i=1}^n\frac{\alpha_i^2}{x_i}+\left(\frac{1}{1-\sum_{i=1}^nx_i}+h''\left(\sum_{i=1}^nx_i\right)\right)\left(\sum_{i=1}^n\alpha_i\right)^2.
$$
As for the determinant of $\hess(\s)$ it is given by
$$
\det\hess(\s)=\det\left(\frac{e_1}{x_1}+U(t)v,\cdots,\frac{e_n}{x_n}+U(t)v\right)
$$
where \begin{itemize}
\item  $e_i$ denotes the vector in $\bbR^n$ which has all entries zero except for the $i$th entry which is $1$, 
\item $v=(1,\cdots,1)$, 
\item $t=\sum_{i=1}^nx_i$ 
\item and 
$$
U(t)=\frac{1}{1-t}+h''(t).
$$
\end{itemize}
It follows from the multilinearity of the determinant that 
$$
\det\hess(\s)=\frac{1}{\prod_{i=1}^nx_i}+U(t)\sum_{i=1}^n\frac{1}{\prod_{j=1,\cdots,n,j\ne i}x_j}=\frac{1+U(t)t}{\prod_{i=1}^nx_i}.
$$
Now 
$$
1+U(t)t=\frac{1}{1-t}+th''(t).
$$
Replacing in (\ref{spec_inv}) the result follows.
\end{proof}
We see that in our setting $\F(\s)$ really just depends on our unknown function $h$. To recover the metric we need only recover $h$ and we will show that the expression (\ref{Fh}) for all $\rho \in \Cinf(\bbR_0^+)$ and $\alpha\in \bbR^n$ completely determines $h$. We are now in a position to prove out main Theorem \ref{main}. 
\begin{proof}[proof of Theorem \ref{main}]  Choose $\alpha=(1,-1,0,\cdots 0)$ in (\ref{Fh}). This is only possible if $n\geq 2$. The argument of $\rho$ i.e. $\alpha^t\hess(\s)\alpha$ becomes $\frac{1}{x_1}+\frac{1}{x_2}$. Expression (\ref{Fh}) simplifies to yield
\begin{equation}\label{Fh11}
\F(\s)=\int_{\cp}\rho\left(\frac{1}{x_1}+\frac{1}{x_2}\right)\sqrt{\frac{\frac{1}{1-\sum_{i=1}^nx_i}+\sum_{i=1}^nx_ih''\left(\sum_{i=1}^nx_i\right)}{\prod_{i=1}^nx_i}}.
\end{equation}
The upshot is that we will be able to simplify the integral in expression (\ref{Fh}) using a change of variable. We will change variables by setting 
$$
\begin{cases}
&\nu=\frac{1}{x_1}+\frac{1}{x_2}\\
&\mu_k=\sum_{i=1}^k x_i, \, k=2,\cdots n.
 \end{cases}
 $$ 
 The expression above does not a define a change of variables on the whole of $\cp$ since there an obvious symmetry $(x_1,x_2,\cdots, x_n)\mapsto (x_2,x_1,\cdots, x_n)$. But if we restrict to $\cp_+$ where 
{
$$
\cp_+=\{x\in\interior\left(\cp\right): x_1> x_2\},
$$}
one can see it becomes a change of variables whose inverse is 
$$
\begin{cases}
&x_1=\frac{1}{2}\left(\mu_2+\sqrt{\mu_2^2-\frac{4\mu_2}{\nu}}\right)\\
&x_2=\frac{1}{2}\left(\mu_2-\sqrt{\mu_2^2-\frac{4\mu_2}{\nu}}\right)\\
&x_k=\mu_k-\mu_{k-1}, \,k\geq 3,
\end{cases}
$$
on the region 
$$\{(\nu,\mu_2,\cdots,\mu_n):0<\frac{4}{\nu}<\mu_2<1,\, 0<\mu_k<1,\,\mu_{k-1}<\mu_{k}, \, \forall k\in \{3,\cdots,n\} \}.$$ This can be easily verified by noting that $x_1+x_2=\mu_2$ and that  $x_1x_2=\frac{\mu_2}{\nu}$. By symmetry, the integral in expression (\ref{Fh}) is twice the integral over $\cp_+$ of the same integrand. We will express this in the $(\nu,\mu)$ coordinates.
\begin{itemize}
\item We start by calculating the new volume form. The volume form $d\nu\wedge d\mu_2\wedge\cdots\wedge d\mu_n$ in the $x$ coordinates is
\begin{equation*}
\begin{split}
&d\nu\wedge d\mu_2 \wedge ... \wedge d\mu_n = \\
&=-\left(\frac{dx_1}{x_1^2}+\frac{dx_2}{x_2^2}\right)\wedge (dx_1+dx_2) \wedge ... \wedge (dx_1+...+dx_n)\\
&= \left(-\frac{1}{x_1^2}+\frac{1}{x_2^2}\right)dx_1\wedge dx_2  \wedge ... \wedge dx_n\\
\end{split}
\end{equation*}
so that 
$$
dx_1\wedge dx_2  \wedge ... \wedge dx_n=\frac{\mu_2}{\nu^2\sqrt{\mu_2^2-\frac{4\mu_2}{\nu}}}d\nu\wedge d\mu_2\wedge\cdots\wedge d\mu_n.
$$
\item Next we will express the integrand in expression (\ref{Fh11}) as a function of $(\nu,\mu)$. This is straightforward. We have
\begin{equation*}
\begin{split}
\rho\left(\frac{1}{x_1}+\frac{1}{x_2}\right)\sqrt{\frac{\frac{1}{1-\sum_{i=1}^nx_i}+\sum_{i=1}^nx_ih''\left(\sum_{i=1}^nx_i\right)}{\prod_{i=1}^nx_i}}&\\
=\rho(\nu)\sqrt{\frac{\frac{1}{1-\mu_n}+\mu_nh''(\mu_n)}{\frac{\mu_2}{\nu}\prod_{i=3}^n(\mu_i-\mu_{i-1})}}&\\
=\rho(\nu)\sqrt{\frac{V^2(\mu_n)\nu}{{\mu_2}\prod_{i=3}^n(\mu_i-\mu_{i-1})}}
\end{split}
\end{equation*}
where 
\begin{equation}\label{V}
V(\mu_n)=\sqrt{\frac{1}{1-\mu_n}+\mu_nh''(\mu_n)}.
\end{equation}
Here we used the fact that $x_1x_2=\frac{\mu_2}{\nu}$. Note that for $n=2$, we set $\prod_{i=3}^n(\mu_i-\mu_{i-1})=1$.
\item The region of integration can be written down from the fact that, as we said above, the inverse of the coordinate change is defined over 
$$\{(\nu,\mu_2,\cdots,\mu_n):0<\frac{4}{\nu}<\mu_2<1,\, 0<\mu_k<1,\,\mu_{k-1}<\mu_{k}, \, \forall k\in \{3,\cdots,n\} \}.$$
\end{itemize}
Putting all these calculations together we see that integral (\ref{spec_inv}) in our context can be written as expression (\ref{Fh11}) which in turn simplifies to give
\begin{IEEEeqnarray*}{c}
2\int_4^\infty \int_{4/\nu}^1 \int_{\mu_2}^1\cdots \int_{\mu_{n-1}}^1
\rho(\nu)\frac{\mu_2\sqrt{\frac{V(\mu_n)^2\nu}{\mu_2\prod_{j=3}^n(\mu_j-\mu_{j-1})}}}{{\nu^2}\sqrt{\mu_2^2-\frac{4\mu_2}{\nu}}}d\mu_n\cdots d\mu_3d\mu_2d\nu\\
=2\int_4^\infty \rho(\nu)\nu^{-\frac{3}{2}}\int_{4/\nu}^1 \int_{\mu_2}^1\cdots \int_{\mu_{n-1}}^1\frac{V(\mu_n)}{\sqrt{\mu_2-\frac{4}{\nu}}\sqrt{\prod_{j=3}^n(\mu_j-\mu_{j-1})}}d\mu_n\cdots d\mu_3d\mu_2d\nu.\\
\end{IEEEeqnarray*}

Since $\rho$ is any arbitrary, smooth, compactly supported function, we conclude that the following function is spectrally determined for $\nu\geq 4$:
\begin{equation*}
f_u(\nu) = \int_{4/\nu}^1 \int_{\mu_2}^1\cdots 
\int_{\mu_{n-1}}^1
\frac{V(\mu_n)}
{\sqrt{\mu_2-\frac{4}{\nu}}\sqrt{\prod_{j=3}^n(\mu_j-\mu_{j-1})}}
d\mu_n\cdots d\mu_3d\mu_2.
\end{equation*}

Let us now apply a final change of coordinates. Consider $s_i$, $i=1,\cdots ,n$, with $s_1 = 1-4/\nu$ and $s_k = 1-\mu_k$ for $k=2,\cdots n$. We have that:
\begin{equation*}
\int_{0}^{s_1} \int_{0}^{s_2}\cdots 
\int_{0}^{s_{n-1}}
\frac{{V}(1-s_n)}
{\prod_{j=1}^{n-1}\sqrt{(s_j-s_{j+1})}}
ds_n\cdots ds_2
\end{equation*}
is spectrally determined for any $s_1 \in [0,1]$. This, however, corresponds to the successive application of the fractional integral operator $J$ (up to a multiplicative constant), as described in (\ref{Abel}). More precisely this is $J(\cdots (JV(1-s_n))(s_{n-1})\cdots(s_1))$. Since as we said before $J$ is invertible we have have that $V(1-s_n)$ is determined and it follows from the definition of $V$ (see equation \ref{V}) that $h''$ is determined so that our metric is determined.
\end{proof}
%%%%%%%%%%%%%%%%%%%%%%%%%%%%%%%%%%%%%%%%%%%%%%%%%%%%%%%%%%%%%%%%%%%%%%%%%%%%%%%%%%

\end{document}